\documentclass[12pt,letterpaper]{article}

\usepackage{amsmath,amsfonts,amsthm,amssymb}

\swapnumbers 

\newtheorem{lemma}{Lemma}[section]
\newtheorem{corollary}[lemma]{Corollary}
\newtheorem{theorem}[lemma]{Theorem}
\newtheorem{assumptions}[lemma]{Standing Assumptions}
\newtheorem{assumption}[lemma]{Standing Assumption}
\theoremstyle{definition} 
\newtheorem{definition}[lemma]{Definition}

\newcommand{\Nat}{{\mathbb N}}

\newcommand\rationals{{\mathbb Q}}
\newcommand\reals{{\mathbb R}}

\newcommand\hilb{{\mathfrak H}}
\newcommand{\PH}{{\mathbb P}({\mathfrak H})}
\newcommand{\EH}{{\mathbb E}({\mathfrak H})}
\newcommand{\GH}{{\mathbb G}({\mathfrak H})}
\newcommand{\BH}{{\mathbb B}({\mathfrak H})}
\newcommand{\0}{{\mathbf 0}}
\newcommand{\1}{{\mathbf 1}}

\newcommand{\dg}{\sp{\text{\rm o}}}
\newcommand{\st}{{\frac{1}{2}}}

\begin{document}

\title{Abstract Hermitian Algebras I. Spectral Resolution}

\author{David J. Foulis{\footnote{Emeritus Professor, Department of
Mathematics and Statistics, University of Massachusetts, Amherst,
MA; Postal Address: 1 Sutton Court, Amherst, MA 01003, USA;
foulis@math.umass.edu.}}\hspace{.05 in} and Sylvia
Pulmannov\'{a}{\footnote{ Mathematical Institute, Slovak Academy of
Sciences, \v Stef\'anikova 49, SK-814 73 Bratislava, Slovakia;
pulmann@mat.savba.sk. The second author was supported by Research
and Development Support Agency under the contract No. APVV-0071-06,
grant VEGA 2/6088/26 and Center of Excellence SAS, CEPI I/2/2005}}}
\date{}

\maketitle

\begin{abstract}
We refer to the real Jordan Banach algebra of bounded Hermitian
operators on a Hilbert space as a Hermitian algebra. We define an
abstract Hermitian algebra (AH-algebra) to be the directed group of
an e-ring that contains a semitransparent element, has the quadratic
annihilation property, and satisfies a Vigier condition on pairwise
commuting ascending sequences. All of this terminology is explicated
in this article, where we launch a study of AH-algebras. Here we
establish the fundamental properties of AH-algebras, including the
existence of polar decompositions and spectral resolutions, and we
show that two elements of an AH-algebra commute if and only if their
spectral projections commute. We employ spectral resolutions to
assess the structure of maximal pairwise commuting subsets of an
AH-algebra.
\end{abstract}

\medskip



\section{Introduction}  

We shall refer to the real Banach space $\GH$ of bounded Hermitian
operators on a Hilbert space $\hilb$, organized in the usual way
into a partially ordered real vector space, as the \emph{Hermitian
algebra} of $\hilb$. We call $\GH$ an ``algebra" because it is, in
fact, a JW-algebra in the sense of \cite[p. 3]{Top}, and we use the
nonstandard notation $\GH$ because, at least at first, we shall be
focusing on its structure as an \emph{partially ordered additive
abelian group} \cite[p. 1]{Good}.

Our purpose in this article is to introduce and launch a study of a
generalization of $\GH$ which we call an \emph{abstract Hermitian
algebra}, or an \emph{AH-algebra} for short. We derive the basic
properties of an AH-algebra, including the existence polar
decompositions and of spectral resolutions for each of its elements.
In subsequent articles, we shall show that, by analogy with
AW$\sp{*}$-algebras and JW-algebras, AH-algebras admit a
classification into types I, II, and III, and that an appropriate
theory of dimension and symmetries exists for such an algebra.
AH-algebras may be regarded as a class of \emph{quantum structures}
in the sense of \cite{DvPu}.

In the sequel, $\reals$ denotes the ordered field of real numbers
and $\Nat$ is the set of positive integers. Also, $\hilb$ is a
Hilbert space with inner product $\langle{\cdot\mid\cdot}\rangle$,
$\BH$ is the Banach $\sp{*}$-algebra with the uniform operator norm
$\|\cdot\|$ of all bounded linear operators on $\hilb$, and as
mentioned above, $\GH$ is the real Banach space under $\|\cdot\|$ of
all Hermitian operators in $\BH$. As usual, $\GH$ is organized into
a partially ordered real linear space by defining $A\leq B$ for
$A,B\in\GH$ iff $\langle{A\psi\mid\psi}\rangle\leq\langle{B\psi
\mid\psi}\rangle$ for all $\psi\in\hilb$. The zero and identity
operators on $\hilb$ are denoted by $\0,\1\in\GH$, and we denote the
``unit interval" in $\GH$ by $\EH :=\{E\in\GH: \0\leq E\leq \1\}$.
Following G. Ludwig \cite{Lud}, operators $A\in\EH$ are called
\emph{effect operators} on $\hilb$.  The complete atomic
orthomodular lattice (OML) \cite{Kalm} of all (orthogonal)
\emph{projection operators} on $\hilb$ is denoted by $\PH
:=\{P\in\GH: P=P\sp{2}\}$.  We note that
\[
\0,\1\in\PH\subseteq\EH\subseteq\GH\subseteq\BH.
\]
As we proceed, we shall use $\BH$, $\GH$, $\EH$, and $\PH$ to
motivate and illustrate various concepts.

\section{e-Rings} 

The following notion of an e-ring was introduced in \cite{FRwE} and
further studied in \cite{FSRI,FoPuPD} as a generalization of the
pair $(\BH,\EH)$.

\begin{definition} \label{df:e-ring}
An \emph{e-ring} is a pair $(R,E)$ consisting of an associative ring
$R$ with unity $1$ and a subset $E\subseteq R$ of elements called
\emph{effects} such that $0,1\in E$; $e\in E \implies 1-e\in E$; and
the set $E\sp{+}$ consisting of all finite sums $e_1+e_2+\cdots+e_n$
with $e_1,e_2,\ldots,e_n\in E$ satisfies the following conditions:
For all $a,b\in E\sp{+}$,

\smallskip

\begin{tabular}{ll}
(i)\ \  $-a\in E\sp{+}\implies a=0$, & (ii)\  $1-a\in E\sp{+}
 \implies a\in E$,\\
(iii)\  $ab=ba\implies ab\in E\sp{+}$,
&(iv)\  $aba\in E\sp{+}$,\\
(v)\ \ $aba=0\implies ab=ba=0$,\  and
& (vi)\ $(a-b)^2\in E\sp{+}$.
\end{tabular}

\smallskip

\noindent If $(R,E)$ is an e-ring, then the subgroup

\[
G :=\{ a-b: a,b\in E\sp{+}\}=E\sp{+}-E\sp{+}
\]
of the additive group of the ring $R$ is called the {\it directed
group of} $(R,E)$, and $P:=\{p\in G: p=p\sp{2}\}$ is called the set
of \emph{projections} in $G$. The group $G$ is organized into a
partially ordered abelian group with positive cone $E\sp{+}$ by
defining, for all $g,h\in G$, $g\leq h\Leftrightarrow h-g \in
E\sp{+}$.
\end{definition}

It is not difficult to check that $(\BH,\EH)$ is an e-ring, the
partially ordered additive abelian group $\GH$ is its directed
group, and $\PH$ is its set of projections. Fundamental properties
of e-rings are developed in \cite{FRwE}. Further examples of
e-rings, in addition to the prototype $(\BH,\EH)$, as  well as
motivation for the developments that follow can be found in
\cite{FRwE, FSRI, FoPuPD}.

\begin{assumptions} \label{as:e-Ring}
In the sequel, we assume that $(R,E)$ is an e-ring, $E\sp{+}$ is the
set of all sums of finite sequences of effects in $E$, $E\sp{+}$ is
the positive cone for the directed group $G$ of $(R,E)$, and $P$ is
the set of projections in $(R,E)$. To avoid trivialities, we also
assume that $0\not=1$.
\end{assumptions}

We note that $G$ is in fact a \emph{directed group} in the technical
sense that it is generated by its own positive cone $E\sp{+}$
\cite[p. 4]{Good}, and the set $E$ is the \emph{unit interval} in
$G$, i.e., $E=\{e\in G: 0\leq e\leq 1\}$. Also, $1$ is an
\emph{order unit}\footnote{Some authors use the terminology
\emph{strong order unit}.} in $G$ \cite[p. 4]{Good}, i.e., if
$g\in G$, there exists $n\in\Nat$ such that $g\leq n\cdot 1$.
Evidently,
\[
0,1\in P\subseteq E\subseteq E\sp{+}\subseteq G\subseteq R.
\]
We understand that $E\sp{+}$, $E$, and $P$ are partially ordered by
the restrictions of the partial order $\leq$ on $G$. By
\cite[Theorem 2.15]{FRwE}, $P$ is an orthomodular poset (OMP) with
$p\mapsto 1-p$ as the orthocomplementation. Since the mappings
$g\mapsto -g$, $e\mapsto 1-e$, and $p\mapsto 1-p$ are
order-reversing and of period $2$ on $G$, $E$, and $P$,
respectively, there is a \emph{duality principle} whereby properties
of existing suprema in $G$, $E$, or $P$ are converted to properties
of infima and \emph{vice versa}.

In what follows, we focus attention on the directed group $G$, the
unit interval $E\subseteq G$, and the OMP $P\subseteq E$ of
projections---the enveloping ring $R$ is just a convenient
mathematical environment in which to study the triple $P\subseteq
E\subseteq G$, and the detailed structure of $R$ will not concern us
here.

\begin{definition} \label{df:commutativity}
Let $g,h\in G$. We define $gCh$ to mean that $g$ commutes with $h$,
i.e., that $gh=hg$. If $A\subseteq G$, we also define $C(A)$, called
the \emph{commutant of $A$ in} $G$ by $C(A) :=\{g\in G \mid gCa,
\forall a\in A\}$. The set $CC(A) :=C(C(A))$ is called the
\emph{bicommutant of $A$ in} $G$, and if $g\in CC(h) :=CC(\{h\})$,
we say that $g$ \emph{double commutes} with $h$.  We also define
$CPC(g) := C(P\cap C(g))$, so that $h\in CPC(g)$ iff $h$ commutes
with every projection that commutes with $g$.
\end{definition}

In contrast to more common usage, e.g. in operator theory, we use
the notation $C(A)$ and $CC(A)$ only in relation to elements of the
directed group $G$ rather than to arbitrary elements in the
enveloping ring $R$. By Definition \ref{df:e-ring} (iii), if $0\leq
g,h\in G$, then $gCh\Rightarrow gh=hg\in G$\,; however, unless
$0\leq g,h$, we do not assume \emph{a priori} that $gCh\Rightarrow
gh\in G$.\footnote{See Lemma \ref{lm:1/2Lemma} (i) below.}  By the
spectral theorem, if $A\in\GH$, then $CPC(A)= CC(A)$\,; in general
however, even the condition $CPC(g)\subseteq C(g)$ may fail.

\begin{lemma} \label{lm:EProperties}
Let $e,f\in E$, let $p\in P$, and let $g,h\in G$. Then:
\begin{enumerate}

\item If $eCf$, then $0\leq ef\leq e,f\leq 1$ and $0\leq e
 \sp{2}\leq e\leq 1$.

\item $e\leq p\Leftrightarrow e=ep\Leftrightarrow e=pe$ and
 $p\leq e\Leftrightarrow p=pe\Leftrightarrow p=ep$.

\item $pgp,\, php\in G$, and if $g\leq h$, then $pgp\leq php$.
\end{enumerate}
\end{lemma}

\begin{proof} For (i) and (ii), see \cite[Lemma 2.6, Theorem  2.9,
Corollary 2.10]{FRwE}. By \cite[Lemma 2.4 (iv)]{FRwE}, $pgp,\, php
\in G$, and if $g\leq h$, then $0\leq h-g$, so $0\leq p(h-g)p=
php-pgp$ by \cite[Lemma 2.4 (v)]{FRwE}, and (iii) follows.
\end{proof}

Parts (ii) and (iii) of the following theorem are of interest
because they provide conditions \emph{not directly involving
multiplication} for an effect to be a projection. See \cite[Theorem
3.2] {FSRI} for a proof of the theorem.

\begin{theorem} \label{th:proj}
If $p\in E$, then the following conditions are mutually equivalent:
{\rm (i)} $p\in P$. {\rm (ii)} If $e,f,e+f\in E$, then $e,f\leq p
\Rightarrow e+f\leq p$. {\rm (iii)} If $e\in E$ with $e\leq p,1-p$,
then $e=0$. {\rm (iv)} $\exists n,m\in\Nat,\,n\not=m\text{ and }
p\sp{n}=p\sp{m}$.
\end{theorem}

\begin{corollary} \label{co:Psup/infClosed}
Suppose that $\emptyset\not=Q\subseteq P$ and that $Q$ has a
supremum {\rm (}respectively, an infimum{\rm )} $p$ in $G$. Then
$p\in P$ and $p$ is the supremum {\rm (}respectively, the
infimum{\rm )} of $Q$ in $P$.\footnote{In general, the converse of
Corollary \ref{co:Psup/infClosed} fails, i.e., if $Q\subseteq P$ and
the supremum (respectively, the infimum) of $Q$ in $P$ exists, it need
not be the supremum (respectively, the infimum) of $Q$ in $G$.}
\end{corollary}

\begin{proof}
By duality it is sufficient to consider the case in which $p$ is the
infimum of $Q$ in $G$. As $0\leq q$ for all $q\in Q$, we have $0\leq
p$. Choose any $q\sb{0}\in Q$. Then $0\leq p\leq q\sb{0}\leq 1$, so
$p\in E$. To prove that $p\in P$, suppose that $e,f,e+f\in E$ with
$e,f\leq p$. Then, for all $q\in Q$, we have $e,f\leq q$, whereupon
$e+f\leq q$ by Theorem \ref{th:proj} (ii), and it follows that
$e+f\leq p$, whence $p\in P$ by Theorem \ref{th:proj} (ii) again. As
$p\in P$, it is clear that $p$ is the infimum of $Q$ in $P$.
\end{proof}

As we progress, we shall study conditions on $G$, $E$, and $P$ that
are suggested by properties of the prototypes $\GH$, $\EH$, and
$\PH$. Among these are the following.

\begin{definition} \label{df:PropertiesI}
\
\begin{enumerate}
\item If there is an effect $h\in E$ such that $2h=1$, then
 $h$ is unique, and we write $\st :=h$ \cite[Definition 4.1]
 {FSRI}. For reasons elucidated \cite[Section 4]{FSRI}, we call
 $\st$, if it exists, the \emph{semitransparent effect}.

\item $G$ has the \emph{quadratic annihilation} (QA) \emph{property}
 iff, for all $g,h\in G$, $gh\sp{2}g=0\Rightarrow gh=hg=0$.

\item $G$ is \emph{archimedean} \cite[p. 20]{Good} iff, whenever
 $g,h\in G$ and $ng\leq h$ for all $n\in\Nat$, it follows that
 $g\leq 0$.
\end{enumerate}
\end{definition}

Of course, $\st\1$ is the semitransparent effect operator in $\EH$.
If $A,B\in\GH$, then the adjoint of $BA$ is $(BA)\sp{*}=AB$, so
$AB\sp{2}A=(BA)\sp{*}(BA)=\0$ implies that $AB=BA=\0$\,; i.e., $\GH$
has QA.  Clearly, $\GH$ is archimedean.

\begin{lemma} \label{lm:1/2Lemma}
Suppose that $\st\in E$, let $g,h,k\in G$, and let $n\in\Nat$. Then:
{\rm (i)} $gCh\Rightarrow gh=hg\in G$. {\rm (ii)} $ghg\in G$. {\rm
(iii)} $\st(gh+hg)\in G$. {\rm (iv)} $g\sp{n}\in G$. {\rm (v)} If
$G$ is archimedean, then $g\sp{n}=0\Rightarrow g=0$. {\rm (vi)} If
$0\leq k\in C(g)\cap C(h)$, then $g\leq h\Rightarrow gk\leq hk$.
\end{lemma}

\begin{proof}
For (i), (ii), (iii), and (iv), see \cite[Theorem 4.1]{FSRI}\,; for
(v), see \cite[Theorem 4.2]{FSRI}.

\smallskip

\noindent (vi) Assume the hypotheses. Then $0\leq h-g,k$ and
$(h-g)Ck$, so $0\leq (h-g)k=hk-gk$ by Definition \ref{df:e-ring}
(iii), and by part (i), $hk,\,gk\in G$.
\end{proof}

\begin{lemma} \label{lm:QALemma}
Suppose that $G$ has {\rm QA} and let $g,h\in G$. Then
$gh=0\Rightarrow hg=0$.
\end{lemma}

\begin{proof}
By QA, $gh=0\Rightarrow gh\sp{2}g=0\Rightarrow gh=hg=0 \Rightarrow
hg=0$.
\end{proof}

\section{AH-Algebras} 

We maintain Standing Assumptions \ref{as:e-Ring}.

\begin{definition}

\

\begin{enumerate}

\item $G$ has the \emph{Vigier} (V) \emph{property} \cite[Definition 5.1]{FSRI}
 iff every ascending sequence $g\sb{1}\leq g\sb{2}\leq\cdots$
 in $G$ that is bounded above in $G$ has a supremum $g$ in $G$, and
 $g\in CC(\{g\sb{n}: n\in\Nat\})$.

\item $G$ has the \emph{complete Vigier} (complete V) \emph{property}
 iff every ascending net $(g\sb{\alpha})\sb{\alpha\in A}$ in $G$ that
 is bounded above in $G$ has a supremum $g$ in $G$, and $g\in
 CC(\{e\sb{\alpha}: \alpha\in A\})$).

\item $G$ has the \emph{commutative Vigier} (CV) \emph{property}
 iff every ascending sequence $g\sb{1}\leq g\sb{2} \leq\cdots$ of
 pairwise commuting elements of $G$ that is bounded above in $G$ has
 a supremum $g$ in $G$, and $g\in CC(\{g\sb{n}: n\in\Nat\})$.

\item A net $(g\sb{\alpha})\sb{\alpha\in A}$ in $G$ is called
 a \emph{C-net} iff for all $\alpha,\beta\in A$, $\alpha\leq
 \beta\Rightarrow g\sb{\alpha}Cg\sb{\beta}$. We say that $G$ has
 the \emph{complete commutative Vigier} (complete CV) \emph{property}
 iff every ascending C-net $(g\sb{\alpha})\sb{\alpha\in A}$ in $G$
 that is bounded above in $G$ has a supremum $g$ in $G$, and
 $g\in CC(\{g\sb{\alpha}: \alpha\in A\})$.
\end{enumerate}

\end{definition}

An argument originally due to J. Vigier \cite{Vigier} shows that
$\GH$ has the V property \cite[page 263]{RNg}; in fact, by
essentially the same argument, $\GH$ has the complete V property.
Obviously, $\text{complete\ }V\Rightarrow V \Rightarrow CV$ and
$\text{complete\ } V\Rightarrow\text{\ complete\ }CV\Rightarrow CV$.

\begin{theorem}\label{th:1/2+CV} Suppose that $\st\in E$
and $G$ has the CV property. Then:
\begin{enumerate}
\item If $0\leq a\in G$, then $0$ is the infimum in $G$ of the
 sequence $((\frac{1}{2})\sp{n}a)\sb{n\in\Nat}$.
\item $G$ is archimedean.
\end{enumerate}
\end{theorem}

\begin{proof}
(i) As $0\leq a$, the sequence $((\frac{1}{2})\sp{n}a)\sb{n\in
\Nat}$ is descending, bounded below by $0$, and its elements commute
pairwise, so by CV and duality, it has an infimum $c$ in $G$ and
$0\leq c$. Also, $c\leq(\frac{1}{2})\sp{n+1}a$ for all $n\in\Nat$,
whence $2c\leq(\frac{1}{2})\sp{n}a$ for all $n\in\Nat$, so $2c\leq
c$, i.e., $c\leq 0$, and it follows that $c=0$.

\smallskip

(ii) Suppose $g,h\in G$ and $ng\leq h$ for all $n\in\Nat$. As $G$ is
directed, there exist $a,b\in G$ with $0\leq a,b$ and $h=a-b \leq
a$, whence $ng\leq a$ for all $n\in\Nat$. In particular,
$2\sp{n}g\leq a$ for all $n\in\Nat$, and it follows that
$g\leq(\frac{1}{2})\sp{n}a$ for all $n\in\Nat$. Consequently, by
part (i), $g\leq 0$.
\end{proof}

Evidently, our prototype $\GH$ is an AH-algebra as per the following
definition.

\begin{definition} \label{df:AHAlg}
The directed group $G$ of the e-ring $(R,E)$ is an \emph{abstract
Hermitian} (AH) \emph{algebra} iff $\st\in E$, $G$ has the quadratic
annihilation (QA) property, and $G$ has the commutative Vigier (CV)
property.
\end{definition}

\begin{assumption} \label{as:AH-alg}
Henceforth, we assume that the directed group $G$ of $(R,E)$ is an
AH-algebra.
\end{assumption}

\begin{theorem} \label{th:SRE}
Let $e\in E$, let $d :=1-e$, let $d\sb{1} :=\st d$, and define the
sequence $(d\sb{n})\sb{n\in\Nat}$ recursively by $d\sb{n+1}
:=\st(d+(d\sb{n})\sp{2})$ for all $n\in\Nat$. Then $(d\sb{n})
\sb{n\in\Nat}$ is an ascending sequence of pairwise commuting
effects in $E\cap CC(e)$, so by CV it has a supremum $s$ in $G$ and
$s\in CC(\{d\sb{n}: n\in\Nat\})\subseteq CC(e)$. Then
$(1-s)\sp{2}=e$ with $1-s \in CC(e)$.
\end{theorem}

\begin{proof}
The proof is identical to the proof of \cite[Theorem 6.1]{FSRI},
which obviously requires only the CV property, not the full V
property.
\end{proof}

As a consequence of Lemma \ref{lm:1/2Lemma} (v), Theorem
\ref{th:1/2+CV} (ii), and Theorem \ref{th:SRE} together with
\cite[Corollary 6.1 and Theorem 6.4]{FSRI} we have the following.

\begin{theorem} \label{th:SqRoot}
If $0\leq g\in G$, there exists a unique element in $G$, called the
square root of $g$ and denoted by $g\sp{1/2}$, such that $0\leq
g\sp{1/2}$ and $(g\sp{1/2})\sp{2}=g$; moreover, $g\sp{1/2} \in
CC(g)$.
\end{theorem}

By Definition \ref{df:e-ring} (vi), if $g=h\sp{2}$ for some $h\in
G$, then $0\leq g$. Conversely, by Theorem \ref{th:SqRoot}, if
$0\leq g$, then there exists $h\in G$, namely $h=g\sp{1/2}$, such
that $g=h\sp{2}$. Thus, \emph{the positive cone in $G$ consists
precisely of squares of elements of $G$.}

As usual, we say that an element $g\in G$ is \emph{invertible} iff
there is an element $h\in G$ such that $gh=hg=1$. If such an $h$
exists, it is unique, it is called the \emph{inverse} of $g$, and it
is written as $g\sp{-1} :=h$.

\begin{theorem}  \label{th:PosInverse}
Let $g\in G$ with $0\leq g$. Then $g$ is invertible iff there exists
$M\in\Nat$ such that $1\leq Mg$. Moreover, if $g$ is invertible,
then $0\leq g\sp{-1}\in CC(g)$.
\end{theorem}

\begin{proof}
The proof of \cite[Lemma 7.1]{FSRI} goes through as it obviously
requires only the CV property rather than the stronger V property.
\end{proof}

Equipped with the Jordan product $(A,B)\mapsto\frac{1}{2}(AB+BA)$,
our prototype $\GH$ is a Jordan algebra. More generally, we have the
following result.

\begin{theorem} \label{th:Jordan}
$G$ can be organized into an archimedean partially ordered real
vector space and, as such, it is a Jordan algebra with respect to
the Jordan product $(g,h)\mapsto\st(gh+hg)$.
\end{theorem}

\begin{proof}
The full V property is not needed for the proof of \cite[Theorem
7.2]{FSRI}---only CV is required. Thus, $G$ can be organized into a
partially ordered real vector space that is also a Jordan algebra
with the indicated Jordan product, and $G$ is archimedean by Theorem
\ref{th:1/2+CV} (ii).
\end{proof}

\noindent\emph{In the sequel, we understand that $G$ is organized
into a partially ordered real vector space as per Theorem
\ref{th:Jordan}. Moreover, we make routine use of Theorems
\ref{th:SqRoot}, \ref{th:PosInverse}; Lemmas \ref{lm:EProperties},
\ref{lm:1/2Lemma}, \ref{lm:QALemma}; and the following lemma.}

\begin{lemma} \label{lm:lambdaLemma}
If $0\leq g\sb{i}\in G$ for $i=1,2,...,n$, there exists $0<\lambda
\in\reals$ such that $\lambda g\sb{i}\in E$ for $i=1,2,...,n$.
\end{lemma}

\begin{proof} As $1$ is an order unit in $G$, there exists $N\in\Nat$
such that $g\sb{1},g\sb{2},...,g\sb{n}\leq N\cdot 1$.  Let $\lambda
:=1/N$.
\end{proof}

\begin{lemma} \label{lm:MonotoneSupInf}
Let $g\in G$ be the supremum {\rm (}respectively, the infimum{\rm )}
in $G$ of the ascending {\rm (}respectively, descending{\rm )}
sequence $(g\sb{n})\sb{n\in\Nat}\subseteq G$ of pairwise commuting
elements. Suppose $0\leq h\in G$ and $hCg\sb{n}$ for all $n\in\Nat$.
Then $gh=hg$ is the supremum {\rm (}respectively, the infimum{\rm )}
in $G$ of $(g\sb{n}h)\sb{n\in\Nat}$.
\end{lemma}

\begin{proof}
We prove the lemma for an ascending sequence---the result for a
descending sequence then follows by duality. By CV, we have $gCh$,
so $gh=hg\in G$. As $0\leq g-g\sb{1}$ and $0\leq h$, there exists
$0\leq\lambda\in\reals$ such that $\lambda(g-g\sb{1}),\, \lambda
h\in E$. For all $n\in\Nat$, $0\leq g-g\sb{n}\leq g-g\sb{1}$, so
$0\leq\lambda(g-g\sb{n})\leq \lambda(g-g\sb{1})\leq 1$, whence
$\lambda(g-g\sb{n}),\,\lambda h\in E$. Also, $\lambda(g-g\sb{n})C
\lambda h$, whence by Lemma \ref{lm:EProperties} (i), $\lambda
(g-g\sb{n})\lambda h\leq \lambda(g-g\sb{n})$, i.e.,
\[
\lambda(g-g\sb{n})h\leq g-g\sb{n}\text{\ for all\ }n\in\Nat.
\]
As $g\sb{n}\leq g$ and $0\leq h\in C(g\sb{n})\cap C(g)$, Lemma
\ref{lm:1/2Lemma} (vi) implies that $g\sb{n}h\leq gh$ for all
$n\in\Nat$. Suppose $k\in G$ and $g\sb{n}h\leq k$ for all
$n\in\Nat$. We have to show that $gh\leq k$. We have
\[
\lambda(gh-k)\leq\lambda(gh-g\sb{n}h)=\lambda(g-g\sb{n})h \leq
g-g\sb{n}\text{\ for all\ }n\in\Nat,
\]
whence
\[
g\sb{n}\leq g-\lambda(gh-k) \text{\ for all\ }n\in\Nat,
\]
and it follows that $g\leq g-\lambda(gh-k)$. Therefore,
$\lambda(gh-k)\leq 0$, so $gh-k\leq 0$, i.e., $gh\leq k$.
\end{proof}

\begin{theorem} \label{th:MonotonicitySR}
Let $g,h\in G$ with $gCh$ and $0\leq g\leq h$. Then: {\rm (i)\ }
$g\sp{2}\leq h\sp{2}$ and {\rm (ii)\ } $g\sp{1/2}\leq h\sp{1/2}$.
\end{theorem}

\begin{proof}
(i) Follows from \cite[Lemma 2.7 (iii)]{FRwE}.

\smallskip

\noindent (ii) Choose $0<\lambda\in\reals$ such that  $e := \lambda
g\in E$ and $f :=\lambda h\in E$. Then $eCf$, and $e\leq f$. As
$e\sp{1/2}=\lambda\sp{1/2}g\sp{1/2}$ and $f\sp{1/2}=
\lambda\sp{1/2}h\sp{1/2}$, it will be sufficient to prove that
$e\sp{1/2}\leq f\sp{1/2}$. Define
\[
d :=1-e,\ \  c :=1-f,\ \  d\sb{1} :=\st d,\ \  c\sb{1} :=\st c
\]
and by recursion, for all $n\in\Nat$,
\[
d\sb{n+1} :=\st(d+(d\sb{n})\sp{2})\text{\ \ and\ \ } c\sb{n+1}
 :=\st(c+(c\sb{n})\sp{2}).
\]
By Theorem \ref{th:SRE}, $(d\sb{n})\sb{n\in\Nat}$ and
$(c\sb{n})\sb{n\in\Nat}$ have suprema $s$ and $t$, respectively, in
$G$; moreover, $e\sp{1/2}=1-s$ and $f\sp{1/2}=1-t$. As $e\leq f$, we
have $c\leq d$, $c\sb{1}\leq d\sb{1}$, and by part (i) and induction
on $n$, $c\sb{n}\leq d\sb{n}$ for all $n\in\Nat$. Therefore, $t\leq
s$, so $e\sp{1/2}=1-s\leq 1-t=f\sp{1/2}$.
\end{proof}

\section{Carrier Projections} 

We maintain Standing Assumption \ref{as:AH-alg}

\begin{lemma} \label{lm:(1-e)^n}
Let $e\in E$. Then $((1-e)\sp{n})\sb{n\in\Nat}$ is a descending
sequence of pairwise commuting effects in $E$, whence by CV it has
an infimum $q$ in $G$ and $q\in CC(e)$. Moreover, $1-q \in CC(e)\cap
P$, and for all $h\in G$, $eh=0\Leftrightarrow (1-q)h=0$.
\end{lemma}

\begin{proof}
As $1-e\in E$, Lemma \ref{lm:EProperties} (i) implies that
$((1-e)\sp{n}) \sb{n\in\Nat}$ is a descending sequence in $E$, and
it is obvious that the terms of this sequence commute pairwise;
therefore, by CV, it has an infimum $q$ in $G$ and $q\in
CC\{(1-e)\sp{n}: n\in\Nat\}\subseteq CC(e)$. Thus, $1-q\in CC(e)$.
Evidently, $0\leq q\leq 1-e\leq 1$, whence $0\leq q \sp{1/2}\leq 1$
by Theorem \ref{th:MonotonicitySR} (ii), i.e., $q\sp{1/2} \in E$,
and it follows from Lemma \ref{lm:EProperties} (i) that $q=
(q\sp{1/2})\sp{2}\leq q\sp{1/2}$. For every $n\in\Nat$, we have
$q\leq (1-e)\sp{2n}$, so by Theorem \ref{th:MonotonicitySR} (ii)
again, $q\sp{1/2} \leq(1-e)\sp{n}$, and it follows that
$q\sp{1/2}\leq q$. Therefore, $q\sp{1/2}=q$, so $q=q\sp{2}\in P$,
whence $1-q\in P \cap CC(e)$.

Suppose $h\in G$ and $eh=0$. Then $0\leq h\sp{2}$ and $hCe$,
therefore $h\sp{2}C(1-e)\sp{n}$ for all $n\in \Nat$. By Lemma
\ref{lm:MonotoneSupInf}, $h\sp{2}q=qh\sp{2}$ is the infimum in $G$
of the sequence $(h\sp{2}(1-e) \sp{n})\sb{n\in\Nat}$. But
$h\sp{2}(1-e)=h\sp{2}$, and by induction on $n$,
$h\sp{2}(1-e)\sp{n}=h\sp{2}$ for all $n\in\Nat$, so all terms in the
sequence $(h\sp{2}(1-e)\sp{n})\sb{n\in\Nat}$ are equal to $h\sp{2}$,
and it follows that $h\sp{2}q=h\sp{2}$.  Therefore,
$(1-q)h\sp{2}(1-q)=0$, so $(1-q)h=0$ by QA. Thus, $eh=0\Rightarrow
(1-q)h=0$.

Conversely, suppose that $(1-q)h=0$. As $q$ is the infimum in $G$ of
$((1-e)\sp{n})\sb{n\in\Nat}$, we have $q\leq 1-e$, so $e\leq 1-q\in
P$, and it follows that $e=e(1-q)$. Therefore $eh= e(1-q)h=0$, and
we have $eh=0\Leftrightarrow(1-q)h=0$.
\end{proof}

\begin{theorem}  \label{th:CarrierProperty}
For each $g\in G$ there is a uniquely determined projection $g\dg\in
P$ such that, for all $h\in G$, $gh=0\Leftrightarrow g\dg h=0$.
Moreover, $g\dg\in P\cap CC(g)$.
\end{theorem}

\begin{proof}
Let $g\in G$. As $0\leq g\sp{2}$, there exists $0<\lambda\in \reals$
such that $e :=\lambda g\sp{2}\in E$. By Lemma \ref{lm:(1-e)^n},
there is a projection $g\dg\in P\cap CC(e)=CC(g\sp{2})\subseteq
CC(g)$ such that, for all $h\in G$, $eh=0\Leftrightarrow g\dg h=0$.
By QA, for all $h\in G$, we have $gh=0\Rightarrow
g\sp{2}h=0\Rightarrow hg\sp{2}h=0\Rightarrow gh=0$, so
\[
gh=0\Leftrightarrow g\sp{2}h=0\Leftrightarrow\lambda g\sp{2}h=0
 \Leftrightarrow eh=0\Leftrightarrow g\dg h=0.
\]
To prove uniqueness, suppose $p\in P$ and $gh=0\Leftrightarrow ph=0$
for all $h\in G$. Then $g\dg h=0\Leftrightarrow ph=0$ for all $h\in
G$. Putting $h=1-p$, we find that $g\dg(1-p)=0$, i.e., $g\dg=g\dg
p$, so $g\dg\leq p$. By symmetry, $p\leq g\dg$, so $p=g\dg$.
\end{proof}

\begin{definition}  \label{df:CarrierProj}
If $g\in G$, the uniquely determined projection $g\dg$ in Theorem
\ref{th:CarrierProperty} is called the \emph{carrier projection} of
$g$.
\end{definition}

As $\GH$ is an AH-algebra, it follows that each Hermitian operator
$A\in\GH$ has a carrier projection $A\dg\in\PH\cap CC(A)$. In fact,
as is easily seen, $A\dg$ is just the projection onto the orthogonal
complement of the null space of $A$.

In view of Lemma \ref{lm:QALemma}, the carrier projection $g\dg\in
P$ of $g\in G$ is characterized not only by the ``right
annihilation" condition $gh=0\Leftrightarrow g\dg h=0\text{\ for
all\ }h\in G$, but also by the corresponding ``left annihilation"
condition $hg=0\Leftrightarrow hg\dg=0\text{\ for all\ }h\in G$.
Therefore, $G$ has the so-called \emph{carrier property}
\cite[Definition 3.3]{FoPuPD}, and the results of \cite[Section
3]{FoPuPD} are at our disposal.

\begin{lemma} \label{lm:CarrierLemma}
Let $g,h\in G$, $p\in P$, and $e\in E$. Then: {\rm (i)} $g\dg\leq
p\Leftrightarrow gp=pg=g$. {\rm (ii)} $g={g\dg}g=gg\dg$. {\rm (iii)}
$p\leq 1-g\dg \Leftrightarrow gp=pg=0$. {\rm (iv)} $e\dg$ is the
smallest projection $p\in P$ such that $e\leq p$. {\rm (v)}
$(g\dg)\dg= g\dg$. {\rm (vi)} $gh=0 \Leftrightarrow g\dg
h\dg=0\Leftrightarrow g\dg\leq1-h\dg$.
\end{lemma}

\begin{proof}
(i)--(v) follow from \cite[Lemma 3.4]{FoPuPD}. To prove (vi), we
observe that $gh=0\Leftrightarrow g\dg h=0$ and by the ``left
annihilation" condition $g\dg h=0\Leftrightarrow g\dg h\dg=0$.
Moreover, as both $g\dg$ and $h\dg$ are projections, $g\dg
h\dg=0\Leftrightarrow g\dg\leq1-h\dg$.
\end{proof}

\begin{theorem} \label{th:SigmaOML}
$P$ is a $\sigma$-complete orthomodular lattice {\rm (}OML{\rm )}.
Moreover, if $G$ has the complete CV property, then $P$ is a
complete OML.
\end{theorem}

\begin{proof}
That $P$ is an OML follows from \cite[Theorem 3.5]{FoPuPD}. Let
$p\sb{1}\leq p\sb{2}\leq\cdots$ be an ascending sequence in $P$. To
prove that $P$ is $\sigma$-complete, it will be sufficient to show
that $(p\sb{n})\sb{n\in\Nat}$ has a supremum in $P$. By Lemma
\ref{lm:EProperties} (ii), the projections in the sequence
$(p\sb{n})\sb{n\in\Nat}$ commute pairwise, whence by CV,
$(p\sb{n})\sb{n\in\Nat}$ has a  supremum $p$ in $G$. By Corollary
\ref{co:Psup/infClosed}, $p\in P$ and $p$ is the supremum of
$(p\sb{n})\sb{n\in\Nat}$ in $P$.

Suppose $G$ has the complete CV-property, let $Q\subseteq P$, let
${\cal F}$ be the directed set under inclusion of all finite subsets
$F$ of $Q$, and for $F\in{\cal F}$, let $q\sb{F}$ be the supremum in
$P$ of $F$.  Then $((q\sb{F})\sb{F\in{\cal F}})$ is an ascending
C-net in $G$ bounded above by $1$, and (arguing as above), one shows
that its supremum in $G$ belongs to $P$ and is the supremum of $Q$
in $P$.
\end{proof}

\begin{definition} \label{df:CarA}
If $A\subseteq G$, then $p\in P$ is a \emph{carrier projection} for
$A$ iff, for all $h\in G$, the condition $ah=0$ for all $a\in A$ is
equivalent to the condition $ph=0$. Clearly, if $A$ has a carrier
projection $p$, then it is unique, and we shall denote it by $A\dg
:=p$.
\end{definition}

We omit the straightforward proof of the following.

\begin{theorem}
Let $A\subseteq G$. Then $A\dg$ exists iff $\{a\dg: a\in A\}$ has a
supremum $p$ in $P$, in which case $A\dg=p$. Therefore, $P$ is a
complete OML iff every subset $A\subseteq G$ has a carrier
projection $A\dg$.
\end{theorem}

If $p,q\in P$, we denote the supremum and infimum of $p$ and $q$ in
$P$ by $p\vee q$ and $p\wedge q$, respectively.

\begin{lemma} \label{lm:p-q,p+q}
Let $p,q\in P$. Then: {\rm (i)} $p\leq q\Leftrightarrow q-p\in P$.
{\rm (ii)} If $p\leq q$, then $q-p=q\wedge(1-p)$. {\rm (iii)} If
$p+q\in P$, then $pq=qp=0$ and $p+q=p\vee q$. {\rm (iv)} If $pCq$,
then $p\vee q=p+q-pq$, $p\wedge q=pq$, and $p+q= p\vee q+p\wedge q$.
\end{lemma}

\begin{proof}
For (i) and (ii), see \cite[Theorem 2.9 and Corollary 2.14]{FRwE}.
For (iii), see \cite[Theorem 2.11 and Corollary 2.13]{FRwE}. For
(iv), see \cite[Theorem 2.12 and Corollary 2.13]{FRwE}.
\end{proof}

\begin{definition} \label{df:AbsVal}
Let $g\in G$. As $0\leq g\sp{2}$, we can and do define $|g|
:=(g\sp{2})\sp{1/2}$. Also, we define $g\sp{+}= \st(|g|+g)$ and
$g\sp{-}=\st(|g|-g)$.
\end{definition}

\begin{lemma} \label{lm:Prop g+g-}
Let $g\in G$ and let $p :=(g\sp{+})\dg$.  Then:

\smallskip

\begin{tabular}{ll}
\ \ \ {\rm (i)}\ $|g|\sp{2}=g\sp{2}$. &
\ \ \ \ \ \ \ {\rm (ii)}\ $|g|,\, g\sp{+},\, g\sp{-}\in CC(g)$.\\
\ \,{\rm (iii)}\ $g=g\sp{+}-g\sp{-}$. &
\ \ \ \ \ \ \,{\rm (iv)}\ $0\leq |g|=g\sp{+}+g\sp{-}$.\\
\ \ {\rm (v)}\ $g\sp{+}g\sp{-}=g\sp{-}g\sp{+}=0$. &
\ \ \ \ \ \ \,{\rm (vi)}\ $|-g|=|g|$.\\
\ {\rm (vii)}\ $g\sp{-}=(-g)\sp{+}$. &
\ \ \ \ \ {\rm (viii)}\ $g\sp{+}=(-g)\sp{-}$.\\
\ \,{\rm (ix)}\ $p\in CC(g)$ &
\ \ \ \ \ \ \ {\rm (x)} $pC|g|$\\
\ \,{\rm (xi)}\ $pg=g\sp{+}$. &
\ \ \ \ \ \ {\rm (xii)}\ $(1-p)g=-g\sp{-}$.\\
{\rm (xiii)}\ $0\leq p|g|=g\sp{+}$. & \ \ \ \ \ \,{\rm (xiv)}\
$0\leq(1-p)|g|=g\sp{-}$.
\end{tabular}
\end{lemma}

\begin{proof}
(i)--(viii) are obvious. By Theorem \ref{th:CarrierProperty} and
(ii), we have $p\in CC(g\sp{+})\subseteq CC(g)$, proving (ix), and
(x) follows from (ix) and (ii). We have $pg\sp{+}=g\sp{+}$, and
since $g\sp{+} g\sp{-}=0$, we also have $pg\sp{-}=0$; hence (xi) and
(xii) follow from $g=g\sp{+}-g\sp{-}$. Likewise, $p|g|=g\sp{+}$ and
$(1-p)|g|=g\sp{-}$ follow from $|g|=g\sp{+}+g\sp{-}$. Since $0\leq
|g|,\,p,\,1-p$, Definition \ref{df:e-ring} (iii) implies that $0\leq
p|g|=g\sp{+}$ and $0\leq(1-p)|g| =g\sp{-}$, proving (xiii) and
(xiv).
\end{proof}

\begin{corollary} \label{co:Characterize g+g-}
If $g\in G$, then $g\sp{+}$ and $g\sp{-}$ are characterized by the
properties $g=g\sp{+}-g\sp{-}$, $g\sp{+}g\sp{-}=0$, and $0\leq
g\sp{+}+ g\sp{-}$.
\end{corollary}

\begin{proof}
Suppose $a,b\in G$, $g=a-b$, $ab=0$, and $0\leq a+b$. Then
$ab=ba=0$, whence $g\sp{2}=a\sp{2}+b\sp{2}=(a+b)\sp{2}$, and as
$0\leq a+b$, it follows that $a+b=(g\sp{2})\sp{1/2}=|g|$. Therefore,
$g\sp{+}= \st(|g|+g)=\st(a+b+a-b)=a$ and
$g\sp{-}=\st(|g|-g)=\st(a+b-a+b)=b$.
\end{proof}

\section{The Comparability and Polar Decomposition Properties} 

We maintain Standing Assumption \ref{as:AH-alg}.

\begin{definition} \label{df:Comparability}
Define $P\sp{\pm}(g):=\{p\in P\cap C(g)\cap CPC(g): (1-p)g\leq 0
\leq pg\}$. We say that $G$ has the \emph{comparability property}
\cite[Definition 2.7] {FoPuPD} iff $P\sp{\pm}(g)\not=\emptyset$ for
all $g\in G$.\footnote{In \cite[Definition 3.4]{FGCOMP} the
comparability property was called \emph{general comparability}
because, for interpolation groups, it is equivalent to the property
of the same name \cite[Chapter 8]{Good}.}
\end{definition}

\begin{theorem} \label{th:Comparability}
If $g\in G$, then $(g\sp{+})\dg\in P\sp{\pm}(g)$, hence $G$ has the
comparability property.
\end{theorem}

\begin{proof}
As $CC(g)\subseteq CPC(g)$, parts (ix) and (xi)--(xiv) of Lemma
\ref{lm:Prop g+g-} imply that $(g\sp{+})\dg\in P\sp{\pm}(g)$.
\end{proof}

In general, there may be more than one projection in $P\sp{\pm}(g)$,
but it can be shown that $(g\sp{+})\dg$ is the smallest such
projection \cite[Theorem 3.1]{FSR}. Moreover, no matter which
projection $p\in P\sp{\pm}(g)$ is chosen, one always has
$g\sp{+}=pg$ and $g\sp{-}=-(1-p)g$ \cite[Theorem 3.2]{FGCOMP}.

By \cite[Corollary 4.6]{FCPOAG}, $G$ is a so-called
\emph{compressible group} \cite[Definition 3.3]{FCG}, and since it
has the comparability property, it is a so-called \emph{comgroup}
\cite[Definition 1.1]{FSR}. Translating \cite[Definition
6.1]{FGCOMP} to our present context, we observe that $G$ has the
\emph{Rickart projection property} iff, for each $g\in G$, there
exists $g\,'\in G$ such that, for all $p\in P$, $p\leq
g\,'\Leftrightarrow pg=gp=0$. By Lemma \ref {lm:CarrierLemma} (iii),
$G$ has the Rickart projection property with $g\,' :=1-g\dg$ and
$g\,'\,'=g\dg$ for all $g\in G$. Therefore, $G$ is a so-called
\emph{Rickart comgroup} \cite[Definition 1.1]{FSR}, whence by
changing notation from $g\,'$ to $1-g\dg$ and from $g\,'\,'$ to
$g\dg$, we can invoke all the results of \cite{FSR} and
\cite{FGCOMP}.

\begin{lemma} \label{lm:LemmaA}
Let $g,h\in G$. Then: {\rm (i)} If $h\in CPC(g)$ and $g\leq h$, then
$g\sp{+}\leq h\sp{+}$.  {\rm (ii)} If $0\leq g\leq h$, then
$g\dg\leq h\dg$. {\rm (iii)} If $h\in CPC(g)$ and $g\leq h$, then
$(g\sp{+})\dg\leq (h\sp{+})\dg$. {\rm (iv)} If $(g\sp{+})\dg =1$,
then $0\leq g$. {\rm (v)} $(g\sp{+})\dg\leq(g\sp{+})\dg
\vee(g\sp{-})\dg=(g\sp{+})\dg+(g\sp{-})\dg=g\dg$.
\end{lemma}

\begin{proof}
For (i), see \cite[Lemma 4.4 (i)]{FGCOMP} and for (ii), see
\cite[Lemma 6.2 (vi)]{FGCOMP}. Clearly, (iii) follows from (i) and
(ii). For (iv), see \cite[Theorem 6.5 (v)]{FGCOMP}, and for (v), see
\cite[Theorem 6.5 (ii)]{FGCOMP}.
\end{proof}

\begin{definition} \label{df:Signum}
An element $s\in G$ is called a \emph{signum} of $g$ iff: (i) $s\in
C(g)\cap CPC(g)$, (ii) $0\leq sg=gs\in G$, (iii) $g=s\sp{2}g$, and
(iv) $\forall h\in G$, $gh=0\implies sh=0$. We say that $G$ has the
\emph{polar decomposition} (PD) \emph{property} \cite[Definition
4.3]{FoPuPD} iff every $g\in G$ has a signum $s\in G$.
\end{definition}

\begin{theorem} \label{th:PD}
Let $g\in G$. Then $s :=(g\sp{+})\dg-(g\sp{-})\dg$ is the unique
signum of $g$; hence $G$ has the polar decomposition {\rm (}PD{\rm
)} property. Moreover: {\rm (i)} $s\in CC(g)$. {\rm (ii)}
$g\dg=s\sp{2}$. {\rm (iii)} $|g|=sg=gs$. {\rm (iv)} $g$ has the
``polar decomposition" $g=s|g|=|g|s$. {\rm (v)} $|g|\dg=g\dg$.
\end{theorem}

\begin{proof}
As $G$ has both the carrier and comparability properties,
\cite[Theorem 4.10]{FoPuPD} implies that the signum $s$ of $g$
exists, $s$ is uniquely determined by $g$, and $s=(g\sp{+})\dg
-(g\sp{-})\dg$. By Lemma \ref{lm:Prop g+g-} (ix), $(g\sp{+})\dg \in
CC(g)$. Likewise, by Lemma \ref{lm:Prop g+g-} $(g\sp{-})\dg
=((-g)\sp{+})\dg\in CC(-g)=CC(g)$, and (i) follows.   See
\cite[Lemma 4.4 and Theorem 4.7 (iii)]{FoPuPD} for proofs of (ii),
(iii), and (iv). To prove (v), we note that $gh=0\Rightarrow sgh=0
\Rightarrow |g|h=0\Rightarrow s|g|h=0\Rightarrow gh=0$, so
$gh=0\Leftrightarrow |g|h=0$.
\end{proof}

\begin{theorem} \label{th:Invertibility}
Let $g\in G$. Then the following conditions are mutually equivalent:
{\rm (i)} $g$ is invertible. {\rm (ii)} $|g|$ is invertible. {\rm
(iii)} There exists $0<\lambda\in\reals$ such that $\lambda\cdot
1\leq |g|$. Moreover, if $g\sp{-1}$ exists, then $g\sp{-1}\in CC(g)$
and the signum $s$ of $g$ satisfies $s\sp{2}=1$.
\end{theorem}

\begin{proof}
Let $s$ be the signum of $g$. As $s\in CC(g)$ and $|g|\in CC(g)$,
the desired equivalences follow from Theorem \ref{th:PosInverse} and
the obvious facts that if $g\sp{-1}$ exists, then $|g|\sp{-1}
=sg\sp{-1}$, and if $|g|\sp{-1}$ exists, then
$g\sp{-1}=s|g|\sp{-1}$. Also, if $g\sp{-1}$ exists, it is clear that
if $h\in G$, then $gh=0 \Leftrightarrow h=0$, so $g\dg=1$, and
therefore $s\sp{2}=g\dg=1$ by Theorem \ref{th:PD} (ii).
\end{proof}

\section{States and the {\boldmath $1$}-Norm} 

We maintain Standing Assumption \ref{as:AH-alg}.

\begin{definition} \label{df:state}
If we regard $G$ and $\reals$ as a ordered additive abelian groups,
then an order-preserving group homomorphism $\omega\colon
G\to\reals$ such that $\omega(1)=1$ is called a \emph{state} for $G$
\cite [p. 60]{Good}. We denote the set of all states for $G$ by
$\Omega(G)$, or simply as $\Omega$ if $G$ is understood.
\end{definition}

Note that $\Omega$ is a convex subset of the locally convex real
linear topological space $\reals\sp{G}$ of real-valued functions on
$G$ with the topology of pointwise convergence. Equipped with the
relative topology inherited from $\reals\sp{G}$, $\Omega$ is a
nonempty compact set \cite[Corollary 4.4 and Proposition 6.5]{Good}
called the \emph{state space} of $G$. By \cite[Lemma 6.7]{Good},
every state $\omega\in\Omega$ is a linear functional on the real
linear space $G$.

\begin{theorem} \label{th:OrderDetermining}
$\Omega$ is ``order determining" in the sense that, for $g,h\in
G$, $g\leq h\Leftrightarrow \omega(g)\leq\omega(h)$ for all
$\omega\in\Omega$.
\end{theorem}

\begin{proof} As $G$ is archimedean, \cite[Theorem 4.14]{Good} implies
that $0\leq h-g\Leftrightarrow 0\leq\omega(h-g)=\omega(h)-\omega(g)$
for all $\omega\in \Omega$.
\end{proof}

\begin{definition} \label{df:Norm}
Define the \emph{$1$-norm} $\|\cdot\|\colon G\to\reals\sp{+}$ by
$$\|g\| :=\inf\{\lambda\in\reals: 0\leq\lambda\text{\ and\ }
-\lambda\cdot 1\leq g\leq \lambda\cdot 1\}$$ for all $g\in G$.
\end{definition}

\begin{theorem} \label{th:Norm}
The $1$-norm $\|\cdot\|$ is a norm on the real linear space $G$.
Moreover, for all $g,h\in G$:  {\rm (i)} $\|g\|=\max\{|\omega(g)|:
\omega\in\Omega\}$. {\rm (ii)} $-h\leq g\leq h\Rightarrow\|g\|\leq
\|h\|$. {\rm (iii)} $0\not=p\Rightarrow\|p\|=1$. {\rm (iv)} $\|pgp\|
\leq\|g\|$. {\rm (v)} If $\beta\sb{i},\,\beta\in\reals$, $0\leq
\beta\sb{i}\leq\beta$, and $0\leq u\sb{i}\in G$ for all $i=1,2,...,
n$, then $\|\sum\sb{i=1}\sp{n}\beta\sb{i}u\sb{i}\|\leq \beta\|\sum
\sb{i=1}\sp{n}u\sb{i}\|$.
\end{theorem}

\begin{proof}
That $\|\cdot\|$ is a norm on $G$ as well as properties (i) and (ii)
can be deduced from the results in \cite [pp. 120--121]{Good}. For
(iii) and (iv), see \cite[Theorem 3.3 (viii), (ix)]{FSR}. Let
$0\leq \lambda\in\reals$.  By the hypotheses of (v), $-\beta\sum
\sb{i=1}\sp{n}u\sb{i}\leq 0\leq\sum\sb{i=1}\sp{n}\beta\sb{i}u\sb{i}
\leq\beta\sum\sb{i=1}\sp{n}u\sb{i}$, whence (v) follows from (ii).
\end{proof}

As is well-known, for the archimedean directed group $\GH$, the
$1$-norm coincides with the uniform operator norm.

\begin{theorem} \label{th:NormofProduct}
Let $g,h\in G$. Then: {\rm (i)} $-1\leq g\leq 1\Leftrightarrow
g\sp{2}\leq 1$. {\rm (ii)} $\|g\sp{2}\|=\|g\|\sp{2}$. {\rm (iii)}
$h=|g|\Rightarrow\|h\|=\|g\|$. {\rm (iv)} $gCh\Rightarrow\|gh\|
\leq\|g\|\|h\|$.
\end{theorem}

\begin{proof}
If $-1\leq g\leq 1$, then $0\leq 1-g,1+g$ with $(1-g)C(1+g)$,
whence $0\leq(1-g)(1+g)=1-g\sp{2}$, i.e., $g\sp{2}\leq 1$.
Conversely, $g\sp{2}\leq 1\Rightarrow -1\leq g\leq 1$ follows
from \cite[Lemma 4.3 (iii)]{FSRI}, proving (i). If $0<\lambda
\in\reals$, then by replacing $g$ by $\lambda\sp{-1}g$ in (i),
we deduce that $-\lambda\cdot1\leq g\leq\lambda\cdot 1
\Leftrightarrow g\sp{2}\leq\lambda\sp{2}$, from which
(ii) follows. If $h=|g|$, then $h\sp{2}=g\sp{2}$, so
$\|h\|\sp{2}=\|g\|\sp{2}$ by (ii), and (iii) follows. To
prove (iv), suppose that $gCh$. Then $|g|C|h|$, so by
(iii) we can assume without loss of generality that
$0\leq g,h$. Moreover, we can assume that $g,h\not=0$,
so that $\|g\|,\|h\|\not=0$, and define $e :=\|g\|\sp{-1}g$ and
$f :=\|h\|\sp{-1}h$. Then $e,f\in E$ with $ef=fe$, hence
$0\leq ef\leq 1$ by Lemma \ref{lm:EProperties} (i), and it
follows that $\|ef\|\leq 1$. Therefore, $\|gh\|=
\|g\|\|h\|\|ef\|\leq \|g\|\|h\|$.
\end{proof}

Recall that $G$ is said to be \emph{monotone $\sigma$-complete} iff
every ascending sequence in $G$ that is bounded above in $G$ has a
supremum in $G$ \cite[Chapter 16]{Good}. Thus, if $G$ has property
$V$, it is monotone $\sigma$-complete.

\begin{theorem} \label{th:Banach}
If $G$ is monotone $\sigma$-complete, then it is a real Banach space
under the $1$-norm.
\end{theorem}

\begin{proof}
See \cite[Proposition 3.9]{Hand}.
\end{proof}

\begin{theorem} \label{th:Convergence}
Let $(g\sb{n})\sb{n\in\Nat}$ be a sequence in $G$ and let $g\in G$.
Then:
\begin{enumerate}

\item If $g\sb{n}\rightarrow g$ in the $1$-norm, then for each
 $\omega\in\Omega$ we have $\omega(g\sb{n})\rightarrow\omega(g)$
 in $\reals$.

\item If $g\sb{1}\leq g\sb{2}\leq\cdots$ and $\omega(g\sb{n})
 \rightarrow\omega(g)$ in $\reals$ for each $\omega\in\Omega$,
 then $g$ is the supremum in $G$ of $(g\sb{n})\sb{n\in\Nat}$.

\item If $g\sb{1}\leq g\sb{2}\leq\cdots$ is an ascending sequence
 of pairwise commuting elements in $G$, $g\in G$, and $g\sb{n}
 \rightarrow g\in G$ in the $1$-norm, then $g$ is the supremum in
 $G$ of $(g\sb{n})\sb{n\in\Nat}$ and $g\in CC(\{g\sb{n}: n\in\Nat)\}$.
\end{enumerate}
\end{theorem}

\begin{proof}
(i) Suppose that $g\sb{n}\rightarrow g\in G$ in the $1$-norm and let
$\omega\in\Omega$. Let $\epsilon\in\reals$ with $\epsilon>0$ and
choose $N\in\Nat$ such that, for all $n\in\Nat$, $n\geq N\Rightarrow
\|g\sb{n}-g\|\leq\epsilon$. By Theorem \ref{th:Norm} (i), If $n\geq
N$, then $|\omega(g\sb{n})-\omega(g)|=|\omega(g\sb{n}-g)|
\leq\|g\sb{n}-g\|\leq\epsilon$, whence $\omega(g\sb{n})\rightarrow
\omega(g)$.

\smallskip

\noindent (ii) Assume the hypotheses of (ii) and let $\omega\in
\Omega$. Then $\omega(g\sb{1})\leq\omega(g\sb{2})\leq\cdots$, and
it follows that $\omega(g)$ is the supremum in $\reals$ of the
sequence $(\omega(g\sb{n})\sb{n\in\Nat})$. In particular, for each
$n\in\Nat$, $\omega(g\sb{n})\leq\omega(g)$, and since
$\omega\in\Omega$ is arbitrary, it follows from Theorem
\ref{th:OrderDetermining} that $g\sb{n}\leq g$. To prove that $g$ is
the supremum in $G$ of $(g\sb{n}) \sb{n\in\Nat}$, suppose $h\in G$
and $g\sb{n}\leq h$ for all $n\in\Nat$. Then, for each
$\omega\in\Omega$, we have $\omega(g\sb{n}) \leq\omega(h)$ for
all $n\in\Nat$, whence, $\omega(g)\leq\omega(h)$, and since
$\omega\in\Omega$ is arbitrary, it follows that $g\leq h$.

\smallskip

\noindent (iii) Follows from (i), (ii), and CV.
\end{proof}

\section{Spectral Resolution} 

We maintain Standing Assumption \ref{as:AH-alg} and we denote the
state space of $G$ by $\Omega$.

\begin{definition} \label{df:SpecBds}
If $g\in G$, then the \emph{spectral lower and upper bounds} for $g$
are defined by $L\sb{g} :=sup\{\lambda\in\reals: \lambda\cdot 1 \leq
g\}$ and $U\sb{g} :=inf\{\lambda\in\reals: g\leq \lambda\cdot 1\}$,
respectively.
\end{definition}

\begin{theorem} \label{th:SpecBds}
If $g\in G$, then: {\rm (i)} $-\infty<L\sb{g}\leq U\sb{g}<\infty$.
{\rm (ii)} $\{\omega(g):\omega\in\Omega\}$ is the closed interval
 $[L\sb{g},U\sb{g}]\subseteq\reals$. {\rm (iii)} $\|g\|=max
 \{|L\sb{g}|,|U\sb{g}|\}$. {\rm (iv)} $L\sb{-g}=-U\sb{g}$ and
 $U\sb{-g}=-L\sb{g}$.
\end{theorem}

\begin{proof}
Parts (i) and (ii) follow as in the proof of \cite[Proposition
4.7]{Good}, (iii) follows as in the proof of \cite[Proposition
4.7]{Good}, and (iv) is obvious.
\end{proof}

In \cite[Section 4]{FSR}, we proved that each element $g$ in a
Rickart comgroup has a rational spectral resolution
$(p\sb{g,\lambda}) \sb{\lambda\in\rationals}$. Under our current
stronger hypotheses, we can extend the rational spectral resolution
as follows to obtain a real spectral resolution
$(p\sb{g,\lambda})\sb{\lambda \in\reals}$ for each element $g\in G$.

\begin{definition} \label{df:SpecRes}
Let $g\in G$ and $\lambda\in\reals$. We define
\[
p\sb{g,\lambda} :=1-((g-\lambda\cdot 1)\sp{+})\dg\in P \text{\ and\
} d\sb{g,\lambda} :=1-(g-\lambda\cdot 1)\dg\in P.
\]
The family of projections $(p\sb{g,\lambda})\sb{\lambda\in\reals}$
is called the \emph{spectral resolution} for $g$, and for
$\lambda\in\reals$, $d\sb{g,\lambda}$ is called the \emph
{$\lambda$-eigenprojection} for $g$. If $d\sb{g,\lambda}\not=0$,
then $\lambda$ is an \emph{eigenvalue} of $g$. If $g$ is understood,
we write the spectral resolution for $g$ as $(p\sb{\lambda})
\sb{\lambda\in\reals}$ and we write the family of eigenprojections
for $g$ as $(d\sb{\lambda})\sb{\lambda\in\reals}$.
\end{definition}

\begin{assumptions} In what follows, $g\in G$; $L :=L\sb{g}$ and
$U :=U\sb{g}$ are the spectral bounds for $g$; $(p\sb{\lambda})
\sb{\lambda\in\reals}$ is the spectral resolution of $g$; and
$(d\sb{\lambda})\sb{\lambda\in\reals}$ is the family of
eigenprojections for $g$.
\end{assumptions}

\begin{lemma}  \label{lm:SRminusg}
Let $(q\sb{\lambda})\sb{\lambda\in\reals}$ be the spectral
resolution of $-g$ and let $(c\sb{\lambda})\sb{\lambda\in\reals}$ be
the family of eigenprojections for $-g$. Then, for all
$\lambda\in\reals$, {\rm (i)}
$q\sb{\lambda}=(1-p\sb{-\lambda})+d\sb{-\lambda}=(1-p\sb{-\lambda})
\vee d\sb{-\lambda}$ and {\rm (ii)} $c\sb{\lambda}=d\sb{-\lambda}$.
\end{lemma}

\begin{proof}
By Lemma \ref{lm:Prop g+g-} (vii), we have
\[
1-q\sb{\lambda}=((-g-\lambda\cdot 1)\sp{+})\dg=((-(g-(-\lambda)1))
 \sp{+})\dg=((g-(-\lambda)1)\sp{-})\dg.
\]
Thus, by Lemma \ref{lm:LemmaA} (v),
\[
(1-p\sb{-\lambda})+(1-q\sb{\lambda})=((g-(-\lambda)1)\sp{+})
 \dg+((g-(-\lambda)1)\sp{-})\dg
\]
\[
=(g-(-\lambda)1)\dg=1-d\sb{-\lambda},
\]
whence, by Lemma \ref{lm:p-q,p+q} (iii), $q\sb{\lambda}=
(1-p\sb{-\lambda})+d\sb{-\lambda}=(1-p\sb{-\lambda})\vee
d\sb{-\lambda}$, proving (i).  Finally, it is clear that
$(-h)\dg=h\dg$ for all $h\in G$, so
\[
c\sb{\lambda}=1-(-g-\lambda\cdot 1)\dg=1-(-(g-(-\lambda)1))\dg
=1-(g-(-\lambda)1)\dg=d\sb{-\lambda}.
\]
\end{proof}

\begin{theorem}  \label{th:SpecProps}
For all $\lambda,\mu\in\reals$:
\begin{enumerate}
\item $p\sb{\lambda},\, d\sb{\lambda}\in P\cap CC(g)$ and
 $d\sb{\lambda}Cp\sb{\lambda}$.

\item $p\sb{\lambda}g-\lambda p\sb{\lambda}\leq 0\leq(1-p\sb{\lambda})g-
 \lambda(1-p\sb{\lambda})$.

\item $\lambda\leq\mu\Rightarrow p\sb{\lambda}\leq p\sb{\mu}$ and
 $p\sb{\mu}-p\sb{\lambda}=p\sb{\mu}\wedge(1-p\sb{\lambda})$.

\item $\lambda<\mu\Rightarrow d\sb{\lambda}\leq p\sb{\lambda}\leq
 1-d\sb{\mu}$.

\item $\lambda>U\Rightarrow p\sb{\lambda}=1$, and $\lambda<U\Rightarrow
 p\sb{\lambda}<1$.

\item $\lambda<L\Rightarrow p\sb{\lambda}=0$, and $L<\lambda\Rightarrow
 0<p\sb{\lambda}$.

\item $L=\sup\{\lambda\in\reals: p\sb{\lambda}=0\}$, and $U=\inf
 \{\lambda\in\reals: p\sb{\lambda}=1\}$.

\item If $\lambda\leq \mu$ and $q\in P$ with $q\leq p\sb{\mu}-p
 \sb{\lambda}$, then $\lambda q\leq qgq\leq \mu q$.
\end{enumerate}
\end{theorem}

\begin{proof}
(i) Clearly, $C(g-\lambda\cdot 1)=C(g)$ and $CC(g-\lambda\cdot 1)
=CC(g)$, whence $p\sb{\lambda},\,d\sb{\lambda}\in P\cap CC(g)$ by
Lemma \ref{lm:Prop g+g-} (ix) and Theorem \ref{th:CarrierProperty}.

\smallskip

\noindent (ii) By Theorem \ref{th:Comparability}, $1-p\sb{\lambda}
=((g-\lambda\cdot 1)\sp{+})\dg\in P\sp{\pm}(g-\lambda\cdot 1)$, and
(ii) then follows from the definition of $P\sp{\pm}(g- \lambda\cdot
1)$

\smallskip

\noindent (iii) Assume that $\lambda\leq\mu$. Then $g-\mu\cdot1 \leq
g-\lambda\cdot1$, and $g-\mu\cdot1\in CC(g-\lambda\cdot1)$; hence
$p\sb{\lambda}\leq p\sb{\mu}$ follows from Lemma \ref{lm:LemmaA}
(iii). Thus, $p\sb{\mu}-p\sb{\lambda}=
p\sb{\mu}\wedge(1-p\sb{\lambda})$ by Lemma \ref{lm:p-q,p+q} (ii).

\smallskip

\noindent (iv) By Lemma \ref{lm:LemmaA} (v), we have
$1-p\sb{\lambda}=
 ((g-\lambda\cdot1)\sp{+})\dg\leq(g-\lambda\cdot 1)\dg=1-d\sb
{\lambda}$, whence $d\sb{\lambda}\leq p\sb{\lambda}$. Assume that
$\lambda<\mu$. By (i), $d\sb{\mu}\in CC(g)$ and $p\sb{\lambda}\in
C(g)$, so $d\sb{\mu}Cp\sb{\lambda}$. By (ii),
$gp\sb{\lambda}=p\sb{\lambda} g\leq\lambda p\sb\lambda$, and as the
projection $d\sb{\mu}$ commutes with both $g$ and $p\sb{\lambda}$,
Lemma \ref{lm:1/2Lemma} (vi) implies that
\[
d\sb{\mu}g p\sb{\lambda}\leq d\sb{\mu}(\lambda p\sb{\lambda})
 =\lambda d\sb{\mu}p\sb{\lambda}.
\]
As $d\sb{\mu}=1-(g-\mu\cdot1)\dg$, we have $(g-\mu\cdot1)d
\sb{\mu}=0$, i.e., $\mu d\sb{\mu}=g d\sb{\mu}=d\sb{\mu}g$.
Therefore,
\[
 \mu d\sb{\mu}p\sb{\lambda}=d\sb{\mu}g p\sb{\lambda}\leq
 \lambda d\sb{\mu}p\sb{\lambda}\leq\mu d\sb{\mu}p\sb{\lambda},
\]
whence $\mu d\sb{\mu}p\sb{\lambda}=\lambda d\sb{\mu}p\sb{\lambda}$,
i.e., $(\mu-\lambda)d\sb{\mu}p\sb{\lambda}=0$. But, $\mu-\lambda>
0$, so $d\sb{\mu}p\sb{\lambda}=0$, and it follows that
$p\sb{\lambda} \leq 1-d\sb{\mu}$.

\smallskip

\noindent (v) If $\lambda>U$, there exists $\mu\in\reals$ such that
$\mu<\lambda$ and $g\leq \mu\cdot 1\leq\lambda\cdot 1$, whereupon
$g-\lambda\cdot 1\leq 0$, i.e., $(g-\lambda\cdot 1)\sp{+}=0$, so
$((g-\lambda\cdot 1)\sp{+})\dg=0$, and it follows that
$p\sb{\lambda}=1$. Conversely, if $p\sb{\lambda}=1$, then
$((g-\lambda\cdot 1)\sp{+})\dg=0$, so $(g-\lambda\cdot 1)\sp{+} =0$,
whence $g-\lambda\cdot 1\leq 0$, and it follows that $U\leq\lambda$;
consequently, $\lambda<U\Rightarrow p\sb{\lambda} <1$.

\smallskip

\noindent (vi) Suppose $\lambda < L$. Then there exists $\mu\in
\reals$ such that $\lambda<\mu$ and $\mu\cdot 1\leq g$. Therefore,
$1\leq(\mu-\lambda)1=\mu\cdot 1-\lambda\cdot 1\leq g-\lambda\cdot 1
=(g-\lambda\cdot 1)\sp{+}$, and it follows from Lemma
\ref{lm:LemmaA} (ii) that $1=1\dg\leq((g-\lambda\cdot
1)\sp{+})\dg=1-p\sb{\lambda}$, whence $p\sb{\lambda}=0$. Conversely,
if $p\sb{\lambda}=0$, then $((g-\lambda\cdot 1)\sp{+})\dg=1$, whence
$0\leq (g-\lambda\cdot 1)$, i.e., $\lambda\cdot 1\leq g$, by Lemma
\ref{lm:LemmaA} (iv), whereupon $\lambda\leq L$; consequently,
$L<\lambda\Rightarrow 0<p\sb{\lambda}$.

\smallskip

\noindent (vii) Follows directly from (v) and (vi).

\smallskip

\noindent (viii) Assume the hypotheses. By (iii), $q\leq p\sb{\mu}$
and $q\leq 1-p\sb{\lambda}$; hence $q=qp\sb{\mu}=p\sb{\mu}q$ and
$q=q(1-p\sb{\lambda})=(1-p\sb{\lambda})q$ by Lemma \ref{lm:p-q,p+q}
(i). Also, by (ii),
\[
\lambda(1-p\sb{\lambda})\leq(1-p\sb{\lambda})g\text{\ \ and\ \ }
 p\sb{\mu}g\leq\mu p\sb{\mu}\,{\rm ;}
\]
hence, by Lemma \ref{lm:EProperties} (iii),
\[
\lambda q=q\lambda(1-p\sb{\lambda})q\leq
q(1-p\sb{\lambda})gq=qgq\text{\ \ and}
\]
\[
qgq=qp\sb{\mu}gq\leq q\mu p\sb{\mu}q=\mu q.
\]
Consequently, $\lambda q\leq qgq\leq\mu q$.
\end{proof}

\begin{theorem}  \label{th:SpecTh}
Suppose that
$\lambda\sb{0},\lambda\sb{1},...,\lambda\sb{n}\in\reals$ with
\[
\lambda\sb{0}<L<\lambda\sb{1}<\cdots<\lambda\sb{n-1}<U<\lambda
\sb{n}
\]
and let $\gamma\sb{i}\in\reals$ with $\lambda\sb{i-1}\leq
\gamma\sb{i}\leq\lambda\sb{i}$ for $i=1,2,...,n$.  Define $u\sb{i}
:=p\sb{\lambda\sb{i}}-p\sb{\lambda\sb{i-1}}$ for $i=1,2,...,n$, and
let $\epsilon :=\max\{\lambda\sb{i}-\lambda\sb{i-1}: i=1,2,...,n\}$.
Then:
\[
 u\sb{1},u\sb{2},...,u\sb{n}\in P\cap CC(g),\ \
 \sum\sb{i=1}\sp{n}u\sb{i}=1,\text{\ and\ }\|g-\sum\sb{i=1}\sp{n}
 \gamma\sb{i}u\sb{i}\|\leq\epsilon.
\]
\end{theorem}

\begin{proof}
In the proof, we understand that $i=1,2,...,n$ and that all sums are
from $i=1$ to $i=n$. By parts (i) and (iii) of Theorem
\ref{th:SpecProps}, we have $p\sb{\lambda\sb{i-1}}\leq
p\sb{\lambda\sb{i}}$ with $p\sb {\lambda\sb{i-1}},\,
p\sb{\lambda\sb{i}}\in P\cap CC(g)$, whence $u\sb{i}\in P\cap
CC(g)$. That $\sum u\sb{i}=1$ follows from parts (v) and (vi) of
Theorem \ref{th:SpecProps}. Since $u\sb{i}\in C(g)$, Theorem
\ref{th:SpecProps} (viii) implies that $\lambda\sb{i-1}u\sb{i}\leq
u\sb{i}g\leq\lambda\sb{i}u\sb{i}$ and, adding these inequalities, we
find that $\sum\lambda\sb{i-1}u\sb{i}\leq\sum u\sb{i}g=1\cdot g=
g\leq\sum\lambda\sb{i}u\sb{i}$. The latter inequalities together
with $\lambda\sb{i-1}\leq\gamma\sb{i}\leq\lambda\sb{i}$ and $0\leq
u\sb{i}$ imply that
\[
-\sum(\lambda\sb{i}-\lambda\sb{i-1})u\sb{i}\leq-\sum(\gamma\sb{i}
 -\lambda\sb{i-1})u\sb{i}\leq g-\sum\gamma\sb{i}u\sb{i}
\]
\[
 \leq\sum(\lambda\sb{i}-\gamma\sb{i})u\sb{i}\leq\sum(\lambda\sb{i}
 -\lambda\sb{i-1})u\sb{i},
\]
whence
\[
 \|g-\sum\gamma\sb{i}u\sb{i}\|\leq\|\sum(\lambda\sb{i}
 -\lambda\sb{i-1})u\sb{i}\|\leq \epsilon\|\sum u\sb{i}\|=
 \epsilon\cdot1=\epsilon
\]
parts (ii) and (v) of Theorem \ref{th:Norm} and part (iv) of Theorem
\ref{th:Norm} with $p=1$.
\end{proof}

\begin{theorem} \label{th:SPComutativity}
If $h\in G$, then $hCg\Leftrightarrow hCp\sb{\lambda}$ for all
$\lambda\in\reals$.
\end{theorem}

\begin{proof}
If $hCg$ and $\lambda\in\reals$, then $hCp\sb{\lambda}$ by Theorem
\ref{th:SpecProps} (i). Conversely, suppose that $hCp\sb{\lambda}$
for all $\lambda\in\reals$. Choose and fix $\alpha,\beta\in\reals$
with $\alpha<L$ and $\beta>U$. As usual, a partition of the closed
interval $[\alpha,\beta]\subseteq\reals$ is understood to be a
finite sequence
$\Lambda=(\lambda\sb{i})\sb{i=0,1,2,...,n}\subseteq[\alpha, \beta]$
such that $\alpha=\lambda\sb{0}<\lambda\sb{1}<\cdots\lambda
\sb{n-1}<\lambda\sb{n}=\beta$. The closed interval
$[\lambda\sb{i-1}, \lambda\sb{i}]$ is called the $i$th subinterval
of $\Lambda$ for $i=1,2,...,n$, and we define $\epsilon(\Lambda)
:=\max\{\lambda \sb{i}-\lambda\sb{i-1}: i=1,2,...,n\}$. For the
partition $\Lambda$, we also define $g(\Lambda)
:=\sum\sb{i=1}\sp{n}\lambda\sb{i-1}(p\sb
{\lambda\sb{i}}-p\sb{\lambda\sb{i-1}})$, and we have $\|g-g(\Lambda)
\|\leq \epsilon(\Lambda)$ by Theorem \ref{th:SpecTh}. As $hCp
\sb{\lambda}$ for all $\lambda\in\reals$, we have $hCg(\Lambda)$.

By recursion, we define a sequence $(\Lambda\sb{n})\sb{n\in\Nat}$ of
partitions of $[\alpha,\beta]$ as follows: $\Lambda\sb{1}$ is the
partition $\alpha=\lambda\sb{0}<\lambda\sb{1}=\beta$ having only one
subinterval, namely $[\alpha,\beta]$ itself. From each partition
$\Lambda\sb{n}$, we form the refined partition $\Lambda\sb{n+1}$,
with twice as many subintervals as $\Lambda\sb{n}$, by appending to
the partition $\Lambda\sb{n}$ the midpoints of all its subintervals.
It is clear that $g(\Lambda\sb{1})\leq g(\Lambda \sb{2})\leq\cdots$
and that $g(\Lambda\sb{i})Cg(\Lambda\sb{j})$ for all $i,j\in\Nat$.
Obviously, $\epsilon(\Lambda\sb{n})=(\beta- \alpha)/2\sp{n-1}$,
whence by Theorem \ref{th:SpecTh}, $g(\Lambda\sb{n}) \rightarrow g$
in the $1$-norm $\|\cdot\|$. Therefore, by Theorem
\ref{th:Convergence} (iii), $g$ is the supremum of the ascending
sequence $(g(\Lambda\sb{n}))\sb{n\in\Nat}$ and $g\in
CC(\{g(\Lambda\sb{n}): n\in\Nat\})$; hence $gCh$.
\end{proof}

\begin{corollary} \label{co:SPCommutativity}
Let $g,h\in G$ and let $A\subseteq G$. Then: {\rm (i)} $gCh$ iff
every projection in the spectral resolution of $g$ commutes with
every projection in the spectral resolution of $h$. {\rm (ii)}
$C(C(A)\cap P)=CC(A)$. {\rm (iii)} $CPC(g)=CC(g)$.
\end{corollary}

\begin{proof}
(i) Follows from Theorem \ref{th:SPComutativity}. As $C(A)\cap
P\subseteq C(A)$, we have $CC(A)\subseteq C(C(A)\cap P)$.
Conversely, suppose $g\in C(C(A)\cap P)$, $h\in C(A)$, and
$(p\sb{h,\lambda})\sb{\lambda\in\reals}$ is the spectral
resolution of $h$. Then by Theorem \ref{th:SPComutativity},
$p\sb{h,\lambda}\in C(A)\cap P$, so $gCp\sb{h,\lambda}$ for
every $\lambda\in\reals$, and therefore $gCh$. Consequently,
$C(C(A)\cap P)\subseteq CC(A)$, and (ii) holds.  Putting
$A :=\{g\}$ in (ii), we obtain (iii).
\end{proof}

The following theorem indicates the sense in which the spectral
resolution of $g$ is ``continuous from the right."

\begin{theorem} \label{th:RightContinuity}
If $\alpha\in\reals$, then $p\sb{\alpha}$ is the infimum in the {\rm
OML} $P$ of $A :=\{p\sb{\mu}: \alpha<\mu\in\reals\}$.
\end{theorem}

\begin{proof}
By Theorem \ref{th:SpecProps} (iii), $p\sb{\alpha}$ is a lower bound
for $A$. Suppose that $r\in P$ is another lower bound for $A$. We
have to prove that $r\leq p\sb{\alpha}$. Evidently, $p\sb{\alpha}
\vee r$ is a lower bound for $A$. Define $q :=(p\sb{\alpha}\vee r)
-p\sb{\alpha}=(p\sb{\alpha}\vee r)\wedge(1-p\sb{\alpha})$ (Lemma
\ref{lm:p-q,p+q} (ii)). It will be sufficient to prove that $q=0$.
Let $\lambda\in\reals$.  If $\lambda\leq\alpha$, then $p\sb{\lambda}
\leq p\sb{\alpha}\leq p\sb{\alpha}\vee r$, so $p\sb{\lambda}Cq$ by
Lemma \ref{lm:EProperties} (ii). If $\alpha<\lambda$, then
$p\sb{\lambda} \in A$, so $q\leq p\sb{\alpha}\vee r\leq
p\sb{\lambda}$, and again $p\sb{\lambda}Cq$; hence $gCq$ by Theorem
\ref{th:SPComutativity}.

Now suppose that $\lambda<\mu\in\reals$. Then $p\sb{\mu}\in A$, so
$q\leq p\sb{\alpha}\vee r\leq p\sb{\mu}$ and $q\leq 1-p\sb{\alpha}$,
so $q\leq p\sb{\mu}\wedge(1-p\sb{\alpha})= p\sb{\mu}-p\sb{\alpha}$,
and it follows from Theorem \ref{th:SpecProps} (viii) that $\alpha
q\leq qgq=gq=qg\leq \mu q$. Let $\omega\in\Omega$. As $\omega$ is a
linear functional on $G$, we have
\[
\alpha\omega(q)=\omega(\alpha q)\leq\omega(qg)\leq\omega(\mu q)=
 \mu\omega(q),
\]
and since $\mu>\alpha$ is arbitrary, it follows that $\omega(\alpha
q) =\omega(qg)$. By Theorem \ref{th:OrderDetermining}, we conclude
that $\alpha q=qg=gq$. Therefore, $q(g-\alpha\cdot 1)=0$, whence
$q\leq 1-(g-\alpha\cdot 1) \dg=d\sb{\alpha}\leq p\sb{\alpha}$ by
Theorem \ref{th:SpecProps} (iv). But $q\leq 1-p\sb{\alpha}$, so
$q=0$.
\end{proof}

In view of Theorem \ref{th:SpecProps} (v), Theorem
\ref{th:RightContinuity} has the following corollary.

\begin{corollary} \label{co:psubU=1}
$p\sb{U}=1$
\end{corollary}

In the same sense as Theorem \ref{th:RightContinuity}, the
eigenprojection $d\sb{\alpha}$ may be interpreted as the ``jump"
that occurs as $\lambda$ approaches $\alpha$ from the left.

\begin{theorem}
If $\alpha\in\reals$, then $p\sb{\alpha}-d\sb{\alpha}$ is the
supremum in the {\rm OML} $P$ of $B :=\{p\sb{\mu}: \alpha>\mu
\in\reals\}$.
\end{theorem}

\begin{proof}
By Theorem \ref{th:SpecProps} (iv), $d\sb{\alpha}\leq p\sb{\alpha}$,
so $p\sb{\alpha}-d\sb{\alpha}=p\sb{\alpha}\wedge(1-d\sb{\alpha})\in
P$ by Lemma \ref{lm:p-q,p+q} (ii). Let
$(q\sb{\lambda})\sb{\lambda\in\reals}$ be the spectral resolution of
$-g$. By Theorem \ref{th:RightContinuity} and Lemma
\ref{lm:SRminusg}, $q\sb{-\alpha}=(1-p\sb{\alpha})+d\sb{\alpha}$ is
the infimum in $P$ of
\[
\{q\sb{\lambda}:
-\alpha<\lambda\}=\{(1-p\sb{-\lambda}+d\sb{-\lambda}:
-\lambda<\alpha\}=\{(1-p\sb{\mu})+d\sb{\mu}: \mu<\alpha\}{\rm ;}
\]
hence by duality in $P$ (in this case, the De\,Morgan law),
$1-((1-p\sb{\alpha})+d\sb{\alpha})=p\sb{\alpha}-d\sb{\alpha}$ is the
supremum in $P$ of
\[
C
:=\{1-((1-p\sb{\mu})+d\sb{\mu}):\mu<\alpha\}=\{p\sb{\mu}-d\sb{\mu}:
\mu<\alpha\}.
\]
We have to show that $p\sb{\alpha}-d\sb{\alpha}=p\sb{\alpha}\wedge
(1-d\sb{\alpha})$ is also the supremum in $P$ of $B$. If
$\mu<\alpha$, then by parts (iii) and (iv) of Theorem
\ref{th:SpecProps}, $p\sb{\mu} \leq
p\sb{\alpha}\wedge(1-d\sb{\alpha})$, i.e., $p\sb{\alpha}\wedge
(1-d\sb{\alpha})$ is an upper bound for $B$. Suppose that $r\in P$
is another upper bound for $B$. Then, if $\mu<\alpha$, we have
$p\sb{\mu}-d\sb{\mu}\leq p\sb{\mu}\leq r$, i.e., $r$ is an upper
bound for $C$; hence $p\sb{\alpha}-d\sb{\alpha}\leq r$, so
$p\sb{\alpha}-d\sb{\alpha}$ is the supremum of $B$.
\end{proof}

\section{Blocks and C-blocks} 

According to Theorem \ref{th:SigmaOML}, $P$ is a $\sigma$-complete
orthomodular lattice. Recall that elements $p,q\in P$ are called
(\emph{Mackey}) \emph{compatible} iff there are pairwise orthogonal
elements $p_1,q_1,r\in P$ such that $p=p_1\vee r=p_1+r$, $q=q_1\vee
r=q_1+r$.

\begin{lemma}\label{lm:compcom} Two elements $p,q\in P$ are
compatible iff they commute.
\end{lemma}
\begin{proof} Let $pq=qp$ and put $r=pq$. By Lemma 4.8 (iv),
$pq=p\wedge q\in P$. Then $p=(p-r)+r, q=(q-r)+r$, and $r,p-r,q-r$
are pairwise orthogonal.

Conversely, let $p=p_1+r, q=q_1+r$ with $p_1,q_1,r$ pairwise
orthogonal. Then $pq=qp=r$.
\end{proof}

A subset $B$ of $P$ is called a \emph{block} of $P$ if $B$ is a
maximal set of pairwise compatible elements \cite[Ch.\,1, \S4]{Kalm}.
In view of Lemma \ref{lm:compcom}, it is clear that $B\subseteq P$
is a block of $P$ iff $B=C(B)\cap P$. It is well known that every
block in $P$ is a maximal Boolean $\sigma$-subalgebra of $P$.
Following \cite[Def. 5.1]{FoPu1}, a subgroup of $G$ having the
form $C(B)$, where $B$ is a block in $P$, will be called a
\emph{C-block} in $G$.

\begin{theorem}\label{th:maxcom} A subset $H$ of $G$ is a C-block
of $G$ iff $H$ is a maximal set of pairwise commuting elements of
$G$.
\end{theorem}

\begin{proof}
If $H\subseteq G$, it is clear that $H$ is a maximal set of pairwise
commuting elements of $G$ iff $H=C(H)$. Suppose $H=C(B)$ for some
block $B=C(B)\cap P$ of $P$. Then by Corollary \ref{co:SPCommutativity}
(ii), $H=C(B)=C(C(B)\cap P)=CC(B)=C(H)$. Conversely, suppose $H=
C(H)$ and put $B :=H\cap P=C(H)\cap P$. Then $CC(H)=C(H)=H$ and,
again by Corollary \ref{co:SPCommutativity} (ii), $B=H\cap P=
CC(H)\cap P=C(C(H)\cap P))\cap P=C(B)\cap P$.
\end{proof}

As a consequence of Theorem \ref{th:maxcom}, $G$ is covered by
its own C-blocks. Moreover, as we proceed to show, each C-block $H$
in $G$ is itself an AH-algebra that has the structure of an
archimedean lattice-ordered commutative real Banach algebra.

\begin{theorem} \label{th:CBlockAHAlg}
Let $H$ be a C-block in $G$. Then $(R,E\cap H)$ is an e-ring with
directed group $H$, the e-ring partial order on $H$ is the
partial order induced from $G$, the set of projections in $H$
is a block $B$ in $P$, $H=C(B)$, and $H$ is an AH-algebra with the
Vigier property. Moreover, under the $1$-norm $\|\cdot\|$, $H$
is a commutative and associative real Banach algebra with unity
element $1$ and with the property that $h\in H\Rightarrow\|h\sp{2}\|
=\|h\|\sp{2}$.
\end{theorem}

\begin{proof}
By definition of a C-block, there exists a block $B$ in $P$
such that $H=C(B)$. We omit the straightforward verification
that $(R,E\cap H)$ satisfies the conditions in Definition
\ref{df:e-ring}, that $H$ is the directed group of $(R,E\cap H)$,
that $B$ is the set of projections in $H$, and that the e-ring
partial order on $H$ is the restriction to $H$ of the partial order
on $G$. Obviously, $\frac12\in C(B)=H$ and $H$ inherits the QA
property from $G$.

To prove that $H=C(B)$ has the V property, suppose $h\sb{1}
\leq h\sb{2}\leq\cdots$ is an ascending sequence in $H$ that
is bounded above in $H$. Then the sequence is bounded above
in $G$, and by Theorem \ref{th:maxcom}, the elements of the
sequence commute pairwise, hence by CV it has a supremum $h$ in
$G$ and $h\in CC\{h\sb{n} : n\in\Nat\}$. If $p\in B$, then
$pCh\sb{n}$ for all $n\in\Nat$, and therefore $hCp$. Consequently,
$h\in C(B)=H$, so $h$ is the supremum of the sequence
$(h\sb{n})\sb{n\in\Nat}$ in $H$, and $h$ double commutes in
$H$ with the set $\{h\sb{n} : n\in\Nat\}$. Thus, $H$ has the V
property, hence it has the CV property, and therefore $H$ is an
AH-algebra.

Obviously, $H=C(B)$ is closed under multiplication by real
numbers, and if $g,h\in H$, then $gh=hg\in G$ by Theorem
\ref{th:maxcom} and Lemma \ref{lm:1/2Lemma} (i), whence
$gh\in H$. Therefore, $H$ is a commutative and associative
real linear algebra with unity element $1$. By Theorem
\ref{th:NormofProduct} (iv), $H$ is a normed linear algebra
under the $1$-norm. As $H$ has the V property, it is
monotone $\sigma$-complete, whence it is a Banach algebra under
the $1$-norm by Theorem \ref{th:Banach}. By Theorem
\ref{th:NormofProduct} (ii), $\|h\sp{2}\|=\|h\|\sp{2}$ for all
$h\in H$.
\end{proof}

Let ${\mathcal A}$ be a linear algebra over $\reals$. We say that
${\mathcal A}$ is a \emph{partially ordered linear algebra} iff
the additive group of ${\mathcal A}$ is a partially ordered
abelian group, and whenever $0\leq a,b\in{\mathcal A}$ and $0
\leq\lambda\in\reals$, we have $0\leq ab$ and $0\leq\lambda a$.
If a partially ordered linear algebra ${\mathcal A}$ is a
lattice, it is called an \emph{$\ell$-algebra} \cite{Fuc}.

\begin{theorem} \label{th:CBlockLattice}
Let $H$ be a C-block in $P$. Then: {\rm (i)} $h\in H\Rightarrow
|h|,\,h\sp{+},h\sp{-},h\dg\in H$. {\rm (ii)} $H$ is an archimedean
Dedekind $\sigma$-complete $\ell$-algebra with order unit $1$.
{\rm (iii)} If $g,h\in H$, then the infimum and supremum of $g$
and $h$ in $H$ are given by  $g\wedge\sb{H}h=g-(g-h)\sp{+}$ and
$g\vee\sb{H}h=g+(h-g)\sp{+}$. {\rm (iv)} If $h\in H$, then the
spectral resolution and the family of eigenprojections of $h$ are the
same whether calculated in $G$ or in $H$.
\end{theorem}

\begin{proof} There is a block $B$ in $P$ such that $H=C(B)$.

\smallskip

\noindent (i) If $h\in H=C(B)$, then $|h|,\,h\sp{+},h\sp{-}\in H$ by Lemma
\ref{lm:Prop g+g-} (ii), and $h\dg\in H$ by Theorem
\ref{th:CarrierProperty}.

\smallskip

\noindent (ii) Obviously, $H=C(B)$ is an archimedean partially
ordered algebra over $\reals$ and $1$ is an order unit in $H$.
To prove that $H$ is a lattice, let $g,h\in H$ and put $p
:=((g-h)^+)\dg$. Then by (i), $p\in H\cap P=C(B)\cap P=B$, and
by Theorem \ref{th:Comparability}, $p\in P^{\pm}(g-h)$, so
$(1-p)(g-h)\leq 0\leq p(g-h)$ with $(1-p)(g-h), p(g-h)\in H$. Put
$a:=ph+(1-p)g$. Then $a\in H$ and $a\leq g,h$. Suppose $b\in H$ and
$b\leq g,h$. Then $pb\leq ph$ and $(1-p)b\leq (1-p)g$, so $b\leq pb
+(1-p)b\leq a$. Thus $a$ is the infimum of $g$ and $h$ in $H$. The
existence of the supremum of $g$ and $h$ in $H$ is shown dually,
hence $H$ is a lattice. By Theorem \ref{th:CBlockAHAlg}, $H$ has
the V property, therefore it is monotone $\sigma$-complete,
and consequently it is Dedekind $\sigma$-complete by
\cite[Lemma 16.7]{Good}.

\smallskip

\noindent (iii) Let $g,h\in H$. Recall that in a comparability
group, the pseudo-meet $g\sqcap h$ and pseudo-join $g\sqcup h$ are
defined by $g\sqcap h:= g-(g-h)^+$, $g\sqcup h:=g+(h-g)^+$
\cite[Definition 5.2]{FGCOMP}. By (i), $g\sqcap h, g\sqcup h
\in H$, and by \cite[Theorem 5.4 (iv)]{FGCOMP} $g\wedge\sb{H}h=
g\sqcap h$ and $g\vee\sb{H}h=g\sqcup h$.

\smallskip

\noindent (iv) Follows directly from (i) and Definition
\ref{df:SpecRes}.
\end{proof}

In view of Theorem \ref{th:CBlockLattice} (iii), we have the
following.

\begin{corollary}
Suppose that $H\sb{1}$ and $H\sb{2}$ are C-blocks in $G$ and that
$g,h\in H\sb{1}\cap H\sb{2}$. Then the infimum and supremum
of $g$ and $h$ as calculated in $H\sb{1}$ are the same as the
infimum and supremum of $g$ and $h$ as calculated in $H\sb{2}$.
\end{corollary}

\end{document}